\documentclass[11pt]{article}

\usepackage{
	amsmath,
	amsfonts,
	amssymb,
	amsthm,
	hyperref,	
	mathrsfs,	
	graphicx,	
	tikz-cd,    
	}

\usetikzlibrary{arrows,arrows.meta}
\usepackage{enumitem}
\setlist[1]{itemsep=0pt}	

\usepackage[letterpaper, portrait, margin=1.1in]{geometry}
\hypersetup{
	colorlinks=true, 
	urlcolor=blue
	}

\newtheorem{thm}{Theorem}
\newtheorem*{thm*}{Theorem}
\newtheorem{dfn}[thm]{Definition}
\newtheorem{lem}[thm]{Lemma}
\newtheorem{prop}[thm]{Proposition}
\newtheorem{cor}[thm]{Corollary}
\newtheorem{fact}[thm]{Fact}

\theoremstyle{definition}
	\newtheorem{ex}[thm]{Example}
	
	\newtheorem{remark}[thm]{Remark}

\numberwithin{thm}{section}

\def\R{\mathbb{R}}

\def\N{\mathbb{N}}

\def\A{\mathcal{A}}

\def\F{\mathcal{F}}

\def\sse{\subseteq}

\def\dom{\operatorname{dom}}

\def\level{\operatorname{level}}

\newcommand\set[2]{ \left\{ \, #1 \, : \, #2 \, \right\} }
\newcommand\ignore[1]{}

\newcommand\grouppres[2]{ \left\langle \, #1 \, \middle| \, #2 \, \right\rangle }

\title{Almost $\mathcal{R}$-trivial monoids are almost never Ramsey}
\author{
  Khatchatourian, Ivan\\
  \texttt{Toronto}
	\and
  Pawliuk, Micheal\\
  \texttt{Calgary}
}
\date{\today}

\begin{document}
\maketitle

\begin{abstract}
	Recent results have generalized Gowers' Theorem (to Lupini's Theorem \cite{lupini}) and the Furstenberg-Katznelson theorem \cite{solecki}, both infinite dimensional Ramsey Theorems. The framework of \cite{solecki} provides a machine which accepts (almost $\mathcal{R}$-trivial) monoids and outputs Ramsey theorems. The major generalization was to extract monoid actions from these theorems.

	We investigate the other direction, and feed into the machine monoids which appear ``naturally'', and which are not extracted from a Ramsey theorem, such as such as $0$-Hecke monoids and hyperplane face monoids. Most examples of monoids coming from geometry will not be Ramsey monoids. We provide a simple combinatorial condition for checking this that goes through the representation of almost $\mathcal{R}$-trivial monoids as a family of (almost) regressive functions.

	Except in extremely small or trivial cases, the naturally occurring monoids will not be Ramsey.
\end{abstract}

\section{
	\texorpdfstring{
		Introduction
	}{
		\ref{sec:intro}. Introduction
	}
}
\label{sec:intro}

A recent paper \cite{solecki} by Solecki has made a connection between infinitary Ramsey theory and (finite) monoid theory. The classical theorems of Furstenberg-Katznelson and Gowers \cite{gowers1992lipschitz}, as well as a recent result by Lupini \cite{lupini}, were all seen to be special cases of a common result with a common proof technique. The main things that distinguish these results are the monoids that measure simplification in each context, and how strong of a Ramsey result is achieved. Solecki was able to show that the full strength of the expected Ramsey result depended precisely on the structure of a well-studied object associated with the monoid $M$: its partial order of left cosets $X(M)$ (ordered by inclusion). More precisely, the result was equivalent to $X(M)$ being linear. In essence, \cite{solecki} produces a machine that takes in a certain kind of monoid and produces new Ramsey results.

The focus of this paper will be to go in a sort of reverse direction to \cite{solecki}. That paper extracted monoids from known Ramsey results, whereas we will take known monoids and see what sorts of results we get. The known Ramsey results to which this machinery has been applied (like Gowers' theorem, the Furstenberg-Katznelson theorem, and Lupini's theorem \cite{lupini}) typically have rather small monoids associated with them. In the case of Gowers' theorem and Lupini's theorem, the associated monoid actually gives the full strength of the expected Ramsey result (so the monoid is called a \textbf{\emph{Ramsey monoid}}). In the Furstenberg-Katznelson theorem, the monoid is not Ramsey, but we still salvage some amount of the expected Ramsey result.

The natural class of monoids to use in Solecki's Ramsey machine is the class of \textbf{\emph{almost $\mathcal{R}$-trivial}} monoids. These are a slight generalization of \textbf{\emph{$\mathcal{R}$-trivial}} monoids, the latter of which is a well studied class. The more general class is precisely the class for which Solecki's result holds (that a monoid $M$ is Ramsey if and only if $X(M)$ is linear), and was introduced in \cite{solecki}. Despite its recent introduction, it is a very natural class which we make precise in a moment.

Surprisingly (to the authors), almost $\mathcal{R}$-trivial monoids are typically not Ramsey. In fact, these monoids are only Ramsey in extremely rare cases. We make this precise in two ways. First, there is a collection of representation theorems for $\mathcal{R}$-trivial monoids (and slight variations thereof) which are of the flavour ``a monoid can be represented as the collection of all regressive functions on a poset''. It turns out that almost $\mathcal{R}$-trivial monoids have a very natural representation theorem that is a slight variation of the folklore representation theorem for $\mathcal{R}$-trivial monoids. 

\begin{thm*}
	Every finite almost $\mathcal{R}$-trivial monoid can be represented as the monoid of all regressive functions on some finite poset, together with constant functions that achieve a minimal values on that poset. The operation is function composition.
\end{thm*}

We investigate monoids represented in similar ways and give classifications of when they produce a linear $X(M)$. For example:

\begin{thm*}
	Let $M = \mathcal{F}(P)$ be the monoid of all regressive functions on a finite poset $P$. Then $X(M)$ is linear if and only if $P$ is a collection of mutually incomparable points together with at most one chain of length $2$.
\end{thm*}

This is extremely restrictive, so we strengthen the functions to get broader classes of Ramsey monoids. Even in our most extreme example, we still see that it is rare for $M$ to be a monoid. (All terms will be defined in Sections \ref{sec:defns} and \ref{sec:Comparability}.)

\begin{thm*}
	Let $M = \mathcal{F}(L)$ be the monoid of all regressive, order-preserving, $1$-level-Lipschitz functions on the finite lattice $L$. $X(M)$ is linear if and only if $L$ is a chain of length $1,2$ or $3$.
\end{thm*}

The second way in which we make precise the notion that almost-$\mathcal{R}$ trivial monoids are almost never Ramsey is by looking at monoids defined using generators and relations. We look at two classic examples from geometry that are given by reflections: the 0-Hecke monoid associated to a Coxeter group, and the hyperplane face monoids. Even just looking at these two examples, it quickly becomes clear that these geometric monoids will only be Ramsey in trivial cases. 

Our presentation of these monoids is very gentle, as we hope to spark the interest of other experts in Ramsey theory. For the monoid experts, we hope to lay out the Ramsey implications in such a way that you can find applications in your field. In the final section we describe our best intuition of what these monoids are actually measuring in the corresponding Ramsey result.

In Section \ref{sec:defns} we describe all the relevant definitions, results and representation theorems concerning monoids. Section \ref{sec:Comparability} is spent discussing a simple combinatorial way to detect in $M$ when it will be Ramsey, using a standard representation theorem. Sections \ref{sec:two_other_monoids} and \ref{sec:geometric_monoids} concern monoids arising from geometry, not extracted from Ramsey theorems. We conclude with a section about viewing these monoids in terms of information compression, and as an example, look at how the Tetris operation in Gowers' theorem relates to information in Banach spaces.

\section*{Acknowledgements}
\addcontentsline{toc}{section}{Acknowledgements}

The authors would like to thank S. Solecki for presenting his results at the Ramsey Doccourse 2016 in Prague, which the second author attended. Solecki also motivated this project by providing many sources and questions for the authors to look at. The authors would also like to thank J. Ne\v set\v ril and J. Hubi\v{c}ka for organizing the Doccourse that made this meeting possible.

\section{
	\texorpdfstring{
		Preliminary definitions and results
	}{
		\ref{sec:defns}. Preliminary definitions and results
	}
}
\label{sec:defns}

\begin{dfn}
	A set $X$ equipped with a binary operation $\cdot$ is called a \textbf{\emph{monoid}} if $\cdot$ is associative and there is a (necessarily) unique identity element $1 \in X$. In practice we will suppress the $\cdot$ and just write $mn$ rather than $m \cdot n$ when the binary operation in question is apparent from context.
\end{dfn}

\begin{dfn}
	Given a monoid $M$ and an element $m \in M$, we let $mM := \set{mx}{x \in M}$, the left coset corresponding to $m$. Let $X(M)$ be the collection of all left cosets of $M$. We will often think of $X(M)$ as the partial order $(X(M), \sse)$. 

	Define an equivalence relation $\mathcal{R}$ on $M$ via:
	\[
		m \mathcal{R} n \quad \Leftrightarrow \quad mM = nM.
	\]
	($\mathcal{R}$ stands for ``right'', as in right multiplication.)

	Given $m \in M$, the equivalence class $[m]$ of $m$ according to the above relation is called its \textbf{\emph{$\mathcal{R}$-class}}. $M$ is called \textbf{\emph{$\mathcal{R}$-trivial}} if every $\mathcal{R}$-class is a singleton (or in other words if $[m] = \{m\}$ for all $m \in M$). $M$ is called \textbf{\emph{almost $\mathcal{R}$-trivial}} if for each $m \in M$ whose $\mathcal{R}$-class contains more than one element, we have $xm = m$ for all $x \in M$.	
\end{dfn}

The class of $\mathcal{R}$-trivial monoids is well studied (see Chapter 2 of \cite{steinberg}), whereas the class of almost $\mathcal{R}$-trivial monoids was introduced in \cite{solecki} as a natural class to which the relevant monoid Ramsey theorem applies. Almost $\mathcal{R}$-trivial monoids are very close (geometrically) to being $\mathcal{R}$-trivial monoids, in a sense we will see in Theorem \ref{thm:almost_R_trivial_rep}. Before stating it, we recall some concepts relating to partially ordered sets.

\begin{dfn}
	A \textbf{\emph{poset}} $(P, \leq_P)$ is a partially ordered set. We typically denote the order relation simply by $\leq$, and will often conflate the poset and its underlying set $P$.
	
	If $P$ is a poset, function $f: P \rightarrow P$ is \textbf{\emph{regressive}} if $f(p) \leq p$ for all $p \in P$, and is \textbf{\emph{order-preserving}} if $x \leq y$ implies $f(x) \leq f(y)$ for all $x,y \in P$.
	
	Two points $x,y \in P$ are \textbf{\emph{comparable (in $P$)}} if $x \leq y$ or $y \leq x$, and are \textbf{\emph{incomparable}} otherwise. A \textbf{\emph{linear order}} is a poset where every two elements are comparable. A \textbf{\emph{chain}} $C \subset P$ is a linearly ordered subset, and a chain is \textbf{\emph{maximal}} if it is maximal with respect to inclusion among chains. For a finite poset $P$, and $x \in P$, we define the \textbf{\emph{level of $x$ (in $P$)}}, denoted $\level(x)$, to be the maximal length of a chain in $P$ with maximal element $x$. Note that any two distinct points on the same level are necessarily incomparable.
	
	A \textbf{\emph{lattice}} $L$ is a poset in which every two elements $x,y \in L$ have a greatest element below both (called the \textbf{\emph{meet}}, denoted $a \wedge b$) and have a least element above both (called the \textbf{\emph{join}}, denoted $a \vee v$). Our lattices will all be finite, so they will all have a unique maximal element and a unique minimal element. 
\end{dfn}

We mention posets and lattices because of the following two representation theorems.

\begin{thm}[folklore]
	\label{thm:R_trivial_rep}
	Every finite $\mathcal{R}$-trivial monoid can be represented as the monoid of all regressive functions on some finite lattice. Moreover, every such collection of functions is an $\mathcal{R}$-trivial monoid, under function composition.
\end{thm}

The corresponding representation for almost $\mathcal{R}$-trivial monoids adds in constant functions that take on minimal values.

\begin{thm}
	\label{thm:almost_R_trivial_rep}
	Every finite almost $\mathcal{R}$-trivial monoid can be represented as the monoid of all regressive functions on some finite poset, together with constant functions that achieve a minimal values on that poset, under function composition.
\end{thm}

The proof is omitted as it is very similar to the folklore result. From these theorems we see how close almost $\mathcal{R}$-trivial monoids are to $\mathcal{R}$-trivial monoids; they differ only by the inclusion of some (minimal) constant functions. These constant functions are the only functions with the property that makes almost $\mathcal{R}$-trivial monoids special.

In Section \ref{sec:Comparability} we examine the structure of $X(M)$ when $M$ is presented in one of these representations. 

The major result of \cite{solecki} produces a Ramsey statement for every almost $\mathcal{R}$-trivial monoid $M$. Appropriately, $M$ is called \textbf{\emph{Ramsey}} if this Ramsey statement about $M$ is true. Solecki then proves that $M$ is Ramsey if and only if $X(M)$ is linear. If it is not linear, a weaker Ramsey statement can still be recovered. 

Here, we will mostly be concerned with checking when $X(M)$ is linear using the representation theorems above; we will not actually use the Ramsey theorems, though we state the main one here for the sake of completeness, after several preliminary definitions.

\begin{dfn}
	Let $M$ be a finite monoid. A \textbf{\emph{pointed $M$-set}} is a set $X$ equipped with an action of $M$ and a distinguished point $x$ such that $Mx = X$. Given a sequence $(X_n)_{n \in \N}$ of pointed $M$-sets, which we treat as mutually disjoint, let $\langle (X_n) \rangle$ be the set of all finite partial functions $f : \N \to \bigcup_{n \in \N} X_n$ such that $f(n) \in X_n$ for all $n \in \dom(f)$. Given $f, g \in \langle (X_n) \rangle$, we say $f \prec g$ if $\max(\dom(f)) < \min(\dom(g))$. Note that if $f \prec g$, then $f \cup g$ is a well-defined element of $\langle (X_n) \rangle$, making $\langle (X_n) \rangle$ a partial semigroup. Also note that $M$ acts on $\langle (X_n) \rangle$ in the natural coordinatewise manner. 

	We say that $(X_n)_{n \in \N}$ has the \textbf{\emph{Ramsey property}} if for every finite colouring of $\langle (X_n) \rangle$, there is a sequence $B = \{f_i\}_{i \in \N} \sse \langle (X_n) \rangle$, such that
	\begin{itemize}
		\item $f_i \prec f_{i+1}$ for all $i \in \N$;
		\item for each $i \in \N$, the range of $f_i$ includes the distinguished element of $X_n$ for some $n \in \dom(f)$; and
		\item the ``subspace'' $[B]$ generated by $B$ is monochromatic, where
		\[
			[B] := \set{ m_0(f_{i_0}) \cdots m_k(f_{i_k})}{k \in \N, \, m_i \in M \text{ for all } i \text{, and } m_i = 1_M \text{ for at least one } i}.
		\]
	\end{itemize}

	$M$ is called \textbf{\emph{Ramsey}} if every sequence of pointed $M$-sets has the Ramsey property.
\end{dfn}

Finally, the main theorem:

\begin{thm}[Solecki 2017, \cite{solecki}]
	If $M$ is almost $\mathcal{R}$-trivial, then $M$ is Ramsey if and only if $X(M)$ is linear. 
\end{thm}

\section{
	\texorpdfstring{
		Comparability in $X(M)$, when $M$ is represented by functions
	}{
		\ref{sec:Comparability}. Comparability in X(M), when M is represented by functions
	}
}
\label{sec:Comparability}

The function representations of monoids given in Theorems \ref{thm:R_trivial_rep} and \ref{thm:almost_R_trivial_rep} allow us to present a useful characterization of when $X(M)$ is linear. We first present an example that motivated \cite{lupini}, \cite{solecki}, and us.

\subsection{
	\texorpdfstring{
		Catalan monoid
	}{
		\ref{subsec:Catalan}. Catalan Monoid
	}
}
\label{subsec:Catalan}

By convention, we denote $[n] = \{0,1,2, \ldots, n-1\}$.

\begin{dfn}
	Given $n \in \N$, the \textbf{\emph{Catalan monoid}} $C_n$ is the monoid of all of all order-preserving, regressive functions from $[n]$ to itself. The class $I_n \sse C_n$ is the class of all such functions that are additionally $1$-Lipschitz, meaning $i-1 \leq f(i) \leq i$ for all $i \in [n]$.
\end{dfn}

These are called ``Catalan monoids'' because of the well-known result that $|C_n| = \frac{1}{n+1} \binom{2n}{n}$, the $n^\text{th}$ Catalan number. See \cite{DHST} and \cite{MS}, for more historical discussion. In \cite{lupini} the elements of $I_n$ are thought of as \textbf{\emph{generalized Tetris functions}}, in the sense that they generalise the usual Tetris function $T: [n] \to [n]$ given by
\[
	T(k) = 
	\begin{cases}
		k-1 & k > 0 \\
		0 & k = 0
	\end{cases}
\]
In \cite{solecki} it was observed that $X(I_n)$ is linear if and only if $n < 4$. 

We examine $M = I_4$ in detail. 
Let $f, g \in M$ be the following two functions:
\[
   \begin{tikzcd}[row sep=0.5em]
      & f & \\
      3
         \arrow[ddrr, red, mapsto] & &
      3 \\
      2
         \arrow[drr, mapsto] & &
      2 \\
      1
         \arrow[rr, cyan, mapsto] & &
      1 \\
      0
         \arrow[rr, mapsto] & &
      0 \\
   \end{tikzcd}
   \quad \quad \quad
   \begin{tikzcd}[row sep=0.5em]
      & g & \\
      3
         \arrow[drr, red, mapsto] & &
      3 \\
      2
         \arrow[drr, mapsto] & &
      2 \\
      1
         \arrow[drr, cyan, mapsto] & &
      1 \\
      0
         \arrow[rr, mapsto] & &
      0 \\
   \end{tikzcd}
\]

Notice that $f(3) = 1 < 2 = g(3)$ (the red arrows) and $f(1) = 1 > 0 = g(1)$ (the blue arrows). As a consequence of the lemma in the next section, these imply that $f M \not\supseteq g M$ and  $f M \not\sse g M$, and therefore that $X(I_4)$ is not linear.

\subsection{
	\texorpdfstring{
		Function representation and comparability
	}{
		\ref{subsec:comparability}. Function representation and comparability
	}
}
\label{subsec:comparability}

Now we will look at this representation in general, and state a necessary condition for comparability in $X(M)$ in this context.

\begin{lem}
	Let $M$ be a monoid consisting of regressive, order-preserving functions on the finite poset $P$. Then $X(M)$ is not linear if there are $f,g \in M$, and $x,y \in P$ such that $f(x) \not \leq g(x)$ and $f(y) \not \geq g(y)$.
\end{lem}

\begin{proof}
	Assume $f,g,x$ and $y$ are witnesses. We will show that $f = f \cdot \text{id} \notin g M$. Let $h \in M$. Since $h$ is regressive we must have $h(x) \leq x$. Since $g$ is order preserving we must have $g(h(x)) \leq g(x)$. By assumption $f(x) \not \leq g(x)$, so $f \neq gh$. The other direction is analogous.
\end{proof}

If $P$ is not linear, it makes it somewhat unlikely that $X(\mathcal{F}(P))$ will be linear (this will be made precise below). Of course $X(M)$ can be non-linear even when $P$ is linear, as in the case of $I_4$.

\subsection{
	\texorpdfstring{
		Function representation using all regressive functions
	}{
		\ref{subsec:using_all_functions}. Function representation using all regressive functions
	}
}
\label{subsec:using_all_functions}

We now make more precise our heuristic that in this context $X(M)$ tends to be non-linear.

\begin{thm}
	\label{thm:all_regressive_XM_not_linear_R_trivial}
	Let $M = \mathcal{F}(P)$ be the monoid of all regressive functions on the finite poset $P$. Then $X(M)$ is not linear if and only if $P$ contains two distinct, but not necessarily disjoint, ``edges'' $E_1 = \{x_1 < x_2\}$ and $E_2 = \{y_1 < y_2\}$.

	Equivalently, $X(M)$ is linear if and only if $P$ is a collection of mutually incomparable points together with at most one chain of length $2$.
\end{thm}

\begin{proof}
	Suppose first that $P$ contains a chain $\{a < b < c\}$ of length $3$.

	We define two elements $f, g \in M$ as follows. Let $f$ fix all elements of $P$ except that $f(b) = a$. Similarly, let $g$ fix all element of $P$ except that $g(c) = b$. These are clearly regressive functions. We will show that $f = f \cdot \text{id} \notin g M$. Let $h \in M$. Note that $f(c) = c$, but $g(c) = b < c$. Since $h$ is regressive, it is impossible for $h(g(c)) = c$. So $f \neq gh$. A similar argument shows that $g \notin fM$. So $fM$ and $gM$ are incomparable, and therefore $X(M)$ is not linear.

	Now suppose that $P$ has no chains of length greater than $2$. Slight modifications to the above argument will show that $X(M)$ will not be linear if $P$ contains two disjoint edges, or when it contains two edges that only agree at the top point, or when it contains two edges that only agree on the lower point. Note however that the functions created will not necessarily be order-preserving. 

	The only remaining option is that $P$ is a collection of mutually incomparable points together with at most ones chain of length $2$.
\end{proof}

\subsection{
	\texorpdfstring{
		Function representation using all OP functions
	}{
		\ref{subsec:using_all_op_functions}. Function representation using all OP functions
	}
}
\label{subsec:using_all_op_functions}

Since the $\mathcal{R}$-trivial case was so easily resolved, we now look at a subclass of the $\mathcal{J}$-trivial case (which looks at all regressive, order-preserving maps on a lattice).

\begin{thm}
	\label{thm:all_regressive_XM_not_linear}
	Let $M = \mathcal{F}(P)$ a monoid of regressive, order-preserving functions on the finite lattice $P$.

	$X(M)$ is not linear if and only if $P$ contains two incomparable elements, or $P$ has a maximal chain of length at least $3$ .

	Equivalently, $X(M)$ is linear if and only if $P$ is a chain of length $1$ or $2$.
\end{thm}

First a technical lemma.

\begin{lem}
	\label{lem:level_lemma}
	Fix a finite poset $P$ and $x \in P$. If there are incomparable $y,z \leq x$ then there are incomparable $y^\prime \leq y$ and $z^\prime \leq z$ with $\level(y^\prime) = \level(z^\prime)$.
\end{lem}

\begin{proof}[Proof of Lemma \ref{lem:level_lemma}.]
	Without loss of generality, assume $m = \level(y) < \level(z) = n$. Let $C = \{x_1, \ldots, x_n = z\}$ be a chain that witnesses that $\level(z) = n$. Note that $y \notin C$, as $y$ and $z$ are incomparable. Since $\level(y) = \level(z_m)$ the elements are incomparable. So $y^\prime = y$ and $z^\prime = z_m$ are the desired elements.
\end{proof}

To streamline the argument, we introduce a function $\phi_C : P \rightarrow C$, where $C$ is a maximal chain in $P$. Define $\phi_C(x) = \max\{c \in C : c \leq x\}$. This map is well-defined since $P$ is a lattice, and $C$ contains the minimal point of $P$. It is easily checked that $\phi_C$ is regressive, order-preserving and maps onto $C$.

\begin{proof}[Proof of \ref{thm:all_regressive_XM_not_linear}.]
	Suppose that $P$ contains two incomparable elements, so that it must contain two distinct maximal chains $C_1$ and $C_2$. Note that they need not have the same cardinality, but they must both contain the maximal and minimal point of $P$.

	By the proof of Lemma \ref{lem:level_lemma}, there are distinct $y_1 \in C_1$, $y_2 \in C_2$ with the same level. Let $\phi_i := \phi_{C_i}$ as before the proof ($i = 1,2$). Typically, $\phi_1, \phi_2, y_1, y_2$ will be the desired witnesses to the theorem, although occasionally there is a small technical obstacle which is fixed by the following two functions.

	Let $\gamma_i : C_i \rightarrow C_i$ be the identity map except on the interval $[y_1 \wedge y_2, y_i)$, where it maps those points to $y_1 \wedge y_2$. Here the meet is computed in $C_1 \cap C_2$, which is well defined since both chains share the minimal element of $P$.

	It is easy to check that $\gamma_i$ is order-preserving and regressive (on $C_i$), so we get that $\gamma_i \phi_i$ is order-preserving and regressive (on $P$). It is also easy to check that $\gamma_i \phi_i (y_i) = y_i$, but $\gamma_1 \phi_1 (y_2) = \gamma_2 \phi_2 (y_1) = y_1 \wedge y_2$. Thus $\gamma_1 \phi_1, \gamma_2 \phi_2, y_1, y_2$ will be the desired witnesses to the theorem.

	Now suppose that $P$ contains a maximal chain $C$ of length at least $3$. By first applying $\phi_C$, we may assume that $P = C = \{a < b < c < \ldots\}$. Let $f$ be defined by $f(c) = f(b) = b$ and $f(a) = a$ (and the identity everywhere else). Let $g$ be defined by $g(c) = c$ and $g(b) = g(a) = a$ (and the identity everywhere else). Note that these are regressive, order-preserving maps on $C$, and $f(c) < g(c)$, but $f(b) > g(b)$. Thus $X(M)$ is not linear. 

	It is obvious that if $P$ is a chain of length $1$ or $2$ then $X(M)$ is linear.
\end{proof}

\subsection{
	\texorpdfstring{
		Function representation using only some regressive, OP functions
	}{
		\ref{subsec:lipschitz}. Function representation using only some regressive, OP functions
	}
}
\label{subsec:lipschitz}

Of course, a monoid representation does not need to use \textit{all} such regressive, order-preserving functions (for example, $I_n$ only uses the $1$-Lipschitz regressive functions).

We now state a version of Theorem \ref{thm:all_regressive_XM_not_linear} for a more restricted class of functions, that is still broad enough to be interesting.

\begin{dfn}
	A function $f: P \rightarrow P$ is \textbf{\emph{$k$-level-Lipschitz}} if whenever $\level(x) = \level(y)+1$ and $x > y$, then $\vert \level(f(x)) - \level(f(y)) \vert \leq k$.
\end{dfn}

In the case where $P$ is a linear order, this notion collapses to $1$-Lipschitz, so results about $I_n$ (in Section \ref{subsec:Catalan}) will be corollaries of results in this section.

We now present the analogous result of Theorem \ref{thm:all_regressive_XM_not_linear} for $k$-level-Lipschitz functions. Note that $k$-level-Lipschitz is more restrictive than $(k+1)$-level-Lipschitz.

\begin{thm}
	\label{thm:2_Lip_regressive_XM_not_linear}
	Let $k \geq 2$. Let $M = \mathcal{F}(P)$ be the monoid of all regressive, order-preserving, $k$-level-Lipschitz functions on the finite lattice $P$.

	$X(M)$ is linear if and only if $P$ is a chain of length $1$ or $2$.
\end{thm}

\begin{proof}
	We make a small adjustment to the argument presented in the proof of Theorem \ref{thm:all_regressive_XM_not_linear}. Observe that if $P$ has a (maximal) chain $C$ of length at least $3$ then we can use $\phi_C$ to map onto it, and then the map $\gamma$ which fixes the bottom two elements of $C$ and sends every other element of $C$ to the element of $C$ third from the bottom. Then we apply $f$ or $g$ as in the proof. Since $f \gamma \phi_C$, and $g \gamma \phi_C$ map $P$ into the levels $1,2$ and $3$, they must be $2$-level-Lipschitz.

	In the case that there are no chains of length $3$ in $P$, then all maps will be trivially $2$-level-Lipschitz and the argument in the first part of the proof of Theorem \ref{thm:all_regressive_XM_not_linear} will be applicable here.
\end{proof}

The case of $1$-level-Lipschitz is trickier because it excludes some of the maps we used in the previous arguments. 

\begin{thm}
	\label{thm:k_Lip_regressive_XM_not_linear}
	Let $M = \mathcal{F}(P)$ be the monoid of all regressive, order-preserving, $1$-level-Lipschitz functions on the finite lattice $P$.

	$X(M)$ is linear if and only if $P$ is a chain of length $1,2$ or $3$.
\end{thm}

\begin{proof}
	Suppose that $P$ contains incomparable elements. Let $a,b$ be elements on the minimal level that contains incomparable elements. Then $a \wedge b$ is precisely one level below $a$ and $b$ and the set $\{x \in P : x \leq a \wedge b\}$ is linearly ordered.

	Define the map $f_a : P \rightarrow P$ by 
	\[
	  f_a(x) =
	  \begin{cases}
	           x & \text{if $x \leq a \wedge b$} \\
	           a & \text{if $x \geq a$} \\
	  a \wedge b & \text{if otherwise}
	  \end{cases}
	\]
	Define $f_b$ similarly. It is easy to check that both functions are order-preserving and regressive. To see that $f_a$ (for example) is $1$-level-Lipschitz, we note that everything above $a \wedge b$ gets sent to $a$ or $a \wedge b$ which are exactly a level apart. Everything else gets fixed. Finally, note that $f_a(a) = a > a \wedge b = f_b(a)$ and $f_a(b) = a \wedge b < b = f_b(b)$, so $X(M)$ is not linear by Lemma \ref{lem:level_lemma}.

	Now suppose that $P$ is a chain. It is easy to check by hand that $X(M)$ is linear if $P$ is a chain of length $1,2$ or $3$. If $P$ contains at least $4$ elements, then we may use the argument for $I_4$ in Section \ref{subsec:Catalan} to show that $X(M)$ is not linear (after possibly mapping everything above the fourth smallest element to the fourth smallest element).
\end{proof}

Of course, this argument almost entirely subsumes the arguments for Theorem \ref{thm:all_regressive_XM_not_linear} and Theorem \ref{thm:2_Lip_regressive_XM_not_linear} (except for the special case of a chain of length $3$). 

\subsection{
	\texorpdfstring{
		Almost $\mathcal{R}$-trivial
	}{
		\ref{subsec:almost_r_trivial}. Almost R-trivial
	}
}
\label{subsec:almost_r_trivial}

We now explore how to identify when $X(M)$ is not linear in the case where $M$ is an almost $\mathcal{R}$-trivial monoid presented using order-preserving, regressive maps on some poset, and constant maps (as in Theorem \ref{thm:almost_R_trivial_rep}).

\begin{lem}
	Fix a finite poset $P$, and let $M$ be a collection of regressive maps on $P$, let $N$ be a collection of constant maps on $P$ that take minimal values. $X(M \cup N)$ is linear if and only if $X(M)$ is linear.
\end{lem}

\begin{proof}
	If $P$ has a unique minimal point, then $X(M) = X(M \cup N)$, and if $P$ has many minimal points then $n_1 M = n_2 M$ for all $n_1, n_2 \in N$, and this coset will be minimal in $X(M \cup N)$, and the result is clear.
\end{proof}

In other words, the ``almost'' part of an almost $\mathcal{R}$-trivial monoid of regressive functions plays no role in whether it is Ramsey.

\section{
	\texorpdfstring{
		Two other natural monoids
	}{
		\ref{sec:two_other_monoids}. Two other natural monoids
	}
}
\label{sec:two_other_monoids}

We present two well-studied $\mathcal{R}$-trivial monoids arising from reflection groups: the class of 0-Hecke monoids and the class of hyperplane face monoids. The key observation is that these monoids are defined in terms of generators and relations. Except in degenerate cases, the generators will be unrelated in $M$, which will ensure that $X(M)$ is non-linear.

We provide an implementation in Sage that computes the $X(M)$ of 0-Hecke monoids and hyperplane face monoids, \cite{code}.

\subsection{
	\texorpdfstring{
		0-Hecke monoids
	}{
		\ref{subsec:0-hecke}. 0-Hecke monoids
	}
}
\label{subsec:0-hecke}

\subsubsection{
	\texorpdfstring{
		Definition
	}{
		\ref{subsubsec:0-hecke_def}. Definition
	}
}
\label{subsubsec:0-hecke_def}

Our presentation follows those presented in \cite{DHST} and \cite{norton}. These monoids are defined in terms of Coxeter groups, which we'll define first.

\begin{dfn}
	A group $W$ is a \textbf{\emph{Coxeter group}} if it admits a presentation of the following form:
	\[
		W = \grouppres{s_1, \dots, s_n}{(s_i s_j)^{m_{ij}}}. 
	\]
	where $m_{ii} = 1$ for all $1 \leq i \leq n$, and $2 \leq m_{ij} \leq \infty$ for all $i \neq j$. 
	
	These same relations can be expressed in the following form:
	\begin{itemize}
		\setlength{\itemsep}{0pt}
		\item $s_i^2 = 1$ for all $i$. 
		\item ``Braid relation'': $\underbrace{s_i s_j s_i \cdots s_j}_{m_{ij}} = \underbrace{s_j s_i s_j \cdots s_i}_{m_{ij}}$. 
	\end{itemize}
	(It is easy to see these two presentations are equivalent.)
\end{dfn}

Coxeter groups have a strong relationship to reflection groups. Of particular note is that every finite reflection group is a Coxeter group, but this need not be true for infinite reflection groups, see \cite{Speyer}. Notably, the finite symmetric groups $S_n$ are all Coxeter groups; $S_n$ is the finite Coxeter group of type $A_{n-1}$.

\begin{dfn}
	Given a finite Coxeter group $W$ (presented using the same symbols as above), the associated \textbf{\emph{0-Hecke monoid}} is the monoid $H_0(W)$ generated by idempotent generators $\pi_1, \dots, \pi_n$ satisfying the same braid relation as above. 
\end{dfn}

The following is well known (and follows immediately from Corollary \ref{cor:word_ext}). In fact, it is $\mathcal{J}$-trivial, see \cite{DHST}.

\begin{thm}
	For any finite Coxeter group $W$, $H_0(W)$ is $\mathcal{R}$-trivial.
\end{thm}

\subsubsection{
	\texorpdfstring{
		Example computations for $A_2$
	}{
		\ref{subsubsec:0-hecke_computation}. Example computations for A2
	}
}
\label{subsubsec:0-hecke_computation}

For our purposes, even the small Coxeter group (called $A_2$) on generators $a,b$ with $ab = ba$ is interesting.

We look at $H = H_0(W)$ given by generators $\pi_a,\pi_b$ and relation $\pi_a \pi_b = \pi_b \pi_a$. Note that the word $\pi_{aba} = \pi_{bab}$ starts with both $\pi_a$ and $\pi_b$. In the following table we compute the left cosets of each element of $H$, and the accompanying diagram we present a visual representation of the partial order $(X(H), \sse)$ in which each $m H$ is identified with $m$, which is in turn identified with its subscript. 

\bigskip

\begin{table}[h]
\centering
	\begin{tabular}{c@{\qquad \qquad}c}
		\begin{tabular}{r|l}
			$m$                     & $m H$                                                  \\ \hline
			$\text{id}$             & $\{\text{id}, \pi_a, \pi_b, \pi_{ab}, \pi_{ba}, \pi_{aba}\}$ \\
			$\pi_a$                 & $\{\pi_a, \pi_{ab}, \pi_{aba}\}$                             \\
			$\pi_b$                 & $\{\pi_b, \pi_{ba}, \pi_{aba}\}$                             \\
			$\pi_{ab}$              & $\{\pi_{ab}, \pi_{aba}\}$                                    \\
			$\pi_{ba}$              & $\{\pi_{ba}, \pi_{aba}\}$                                    \\
			$\pi_{aba} = \pi_{bab}$ & $\{\pi_{aba}\}$
		\end{tabular}
	&
		$
   		\begin{tikzcd}
      		& \operatorname{id} & \\
      		b 
        	\arrow[ur]
         	& & 
      		a 
         		\arrow[ul] \\
      		ba
         	\arrow[u]
         	& & 
     		ab 
         	\arrow[u] \\
      		&
     		aba = bab
         	\arrow[ul]
         	\arrow[ur] &
   		\end{tikzcd}
		$
	\end{tabular}
\end{table}

\noindent Notice that this $X(H)$ is not linear.


\subsubsection{
	\texorpdfstring{
		Basic results about $X(M)$, for $M=H_0(W)$
	}{
		\ref{subsubsec:0-hecke_basics}. Basic results about X(M), for M=H0(W)
	}
}
\label{subsubsec:0-hecke_basics}

Let $M=H_0(W)$, where $W$ is a finite Coxeter group. $X(M)$ is the partial order of left cosets of $M$, ordered by subset. We draw some general (well-known) conclusions.

\begin{fact}
	For $\pi_x \in M$ with $x\in W$, then $\pi_x M$ is the collection of all $\pi_y$ such that $x$ is an initial subword of $y$.
\end{fact} 

\begin{cor}
	\label{cor:word_ext}
	For $\pi_x,\pi_y$ words in $M$ with $x,y$ words in $W$, then 
	\[
		\pi_x M \sse \pi_y M \Leftrightarrow y \text{ is an initial subword of } x \text{ in } W.
	\]
\end{cor}

This corollary makes precise the notion that when $M$ is defined in terms of generators and relations, $X(M)$ is really just the partial order of words ordered by end extension.

\subsection{
	\texorpdfstring{
		Hyperplane face monoids
	}{
		\ref{subsec:hyperplane_face}. Hyperplane face monoids
	}
}
\label{subsec:hyperplane_face}

This object appears with slightly different definitions in the literature. The most intuitively clear definition is the one in \cite{BRO04}, which we (more or less) reproduce first here. A more formal one from \cite{SAL09} is reproduced next. See also \cite{BHR99}.

\begin{dfn}
	Let $\A$ be a finite collection of hyperplanes in $\R^d$ for some integer $d$ (ie. in a real vector space). $\A$ is called an \textbf{\emph{arrangement}} of hyperplanes. We assume without loss of generality that $\bigcap_{H \in \A} H = \{0\}$. (There is no loss of generality since if the common intersection of the hyperplanes in $\A$ is nontrivial, we can quotient the vector space by this intersection and work in the quotient. Alternatively one can just only define these monoids for arrangements that satisfy this property.)

	The hyperplanes in $\A$ divides $\R^d$ into regions called \textbf{\emph{chambers}}. In $\R^3$, these regions are ``triangular cones'', and therefore have three faces each, which are shaped like two-dimensional sectors. Each of these faces has two lines that form its boundary. Finally, each of these lines has a single point (the origin) that forms its boundary. 

	Let $\F$ be the collection of all the open chambers formed by $\A$, their (relatively) open boundary faces, the (relatively) open boundaries of \emph{those} faces, and so on for all dimensions up to and including the singleton $\{0\}$. The elements of $\F$ are all simply called \textbf{\emph{faces}} (and in particular the chambers themselves are also called faces). 

	Define a binary operation on $\F$ in the following way. Given two faces $A, B \in \F$, define $AB$ to the face one enters first when moving a small positive distance along any straight line from $A$ to $B$.
\end{dfn}

\begin{ex}
	Let $\A$ be three skew lines through the origin in $\R^2$. The semigroup $\F$ corresponding to this arrangement has 13 elements: the 6 chambers, 6 rays, and the origin. The image below illustrates an example of the binary operation. 

	\begin{center}
		\includegraphics[scale=0.6]{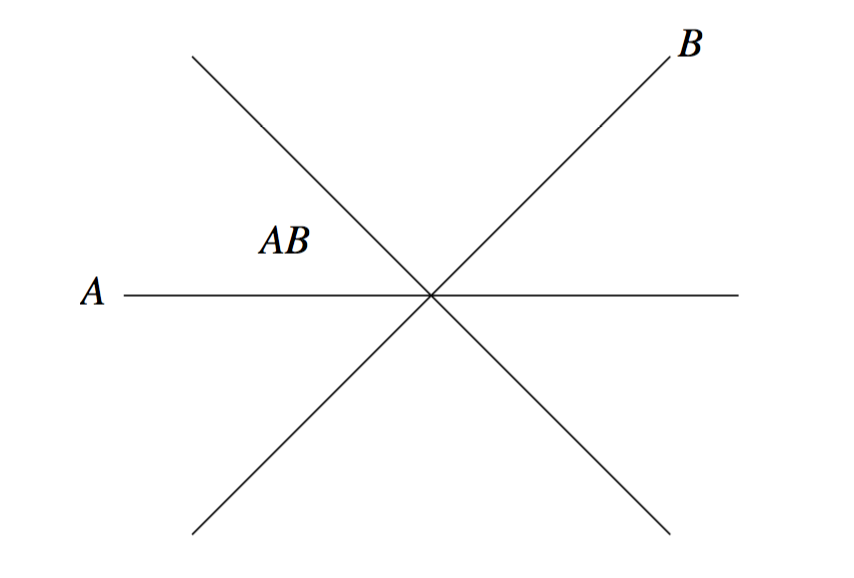}
	\end{center}

	When thinking about this operation, remember that the faces are all (relatively) open. 

	For another example from the same picture, if $C$ is the ray opposite from $A$, then $AC = A$. 
\end{ex}

Now a slightly more formal treatment of the above. 

\begin{dfn}
	Let $\A$ be an arrangement as in the previous definition. We define a \textbf{\emph{face}} to be any set obtained by intersecting, for each $H \in \A$, either $H$ itself or one of the open half-spaces formed by $H$. A more formal way to define this is as follows.

	For each $H \in \A$, let $H^-$ and $H^+$ be the two open half-spaces formed by $H$. The choice of which one is which is arbitrary, but should stay fixed. We also define $H^0 = H$. A \textbf{\emph{face}} $A$ is a nonempty set of the form
	\[
		A = \bigcap_{H \in \A} H^{\sigma_H(A)},
	\]
	where $\sigma_H(A) \in \{-, +, 0\}$. The finite sequence $\langle \sigma_H(A) \rangle_{H \in \A} \in \{-, +, 0\}^\A$ therefore uniquely specifies the face $A$. It is denoted by $\sigma(A)$, and called the \textbf{\emph{sign sequence}} of $A$. Note that the chambers are precisely those faces $A$ such that $\sigma_H(A) \neq 0$ for all $H \in \A$. 

	Then $\F$ is the set of all faces, as determined by the above definition. (Note that this definition doesn't imply that $|\F| = 3^{|\A|}$, since many of these intersections could be empty.) 

	Given this framework, we can define the binary operation algebraically:
	\[
		\sigma_H(AB) = 
		\begin{cases}
			\sigma_H(A) & \text{if } \sigma_H(A) \neq 0 \\
			\sigma_H(B) & \text{if } \sigma_H(A) = 0 
		\end{cases}
	\]
\end{dfn}

\begin{remark}
	A way to think about this operation: take the sign sequences for $A$ and $B$ (in the same order, of course) and stack them on top of each other, with $\sigma(A)$ and top of $\sigma(B)$. If there are any occurrences of $0$ in the top row, replace that entry with the corresponding entry from the bottom row. 

	At the end of this process, the new top row is the sign sequence of $AB$.
\end{remark}

From the above definition in particular it follows that this operation is associative. Also, the origin is the identity element since $\sigma_H(\{0\}) = 0$ for all $H$ by assumption. 

There is a natural partial order on $\F$ as well. Given two faces $A, B \in \F$, we let:
\[
	A \leq B \Longleftrightarrow \text{for all } H \in \A, \text{ either } \sigma_H(A) = 0 \text{ or } \sigma_H(A) = \sigma_H(B).
\]

Equivalently, $A \leq B$ if and only if $A \sse \overline{B}$ (where on the right we mean the topological closure of $B$ in the usual topology on $\R^d$).

This partial order makes good intuitive sense. It essentially says that $A \leq B$ if $A$ is a face of $B$ (in the usual geometric sense). The chambers are maximal elements in this partial order, and the origin is the unique minimal element. 

\begin{prop}
	Some facts:
	\begin{itemize}
		\item $A^2 = A$ for all $A \in \F$. (This is immediate from the more formal definition of the operation above, if it's not clear from the description in the first definition.)
		\item $ABA = AB$ for all $A,B \in \F$. 
		\item $AB = B$ if and only if $A \leq B$. 
	\end{itemize}
\end{prop}

The following is well-known:

\begin{thm}
	Every hyperplane monoid is $\mathcal{R}$-trivial.
\end{thm}

\section{
	\texorpdfstring{
		Linearity of $X(M)$, for geometric monoids
	}{
		\ref{sec:geometric_monoids}. Linearity of X(M), for geometric monoids
	}
}
\label{sec:geometric_monoids}

A simple observation is enough to ensure that $X(M)$ will never be linear when $M$ is defined in terms of generators and relations.

\begin{lem}	
	Suppose that $M$ is a monoid with a representation in terms of at least two generators ($a,b$) and some relations. Then $X(M)$ is not linear, as $aM$ and $bM$ are not comparable.
\end{lem}

\begin{cor}
	A $0$-Hecke monoid is never Ramsey, except in the degenerate case that the underlying Coxeter group has only one generator.
\end{cor}

\begin{cor}
	A (non-trivial) hyperplane face monoid is never Ramsey.
\end{cor}
\section{
	\texorpdfstring{
		Conclusion
	}{
		\ref{sec:conclusion}. Conclusion
	}
}
\label{sec:conclusion}

\subsection{
	\texorpdfstring{
		Overall picture
	}{
		\ref{subsec:overall_picture}. Overall picture
	}
}
\label{subsec:overall_picture}

In an effort to understand $X(M)$ we have looked at naturally occurring almost $\mathcal{R}$-trivial monoids $M$ from two perspectives. In Section \ref{sec:Comparability} we leveraged function representation theorems, and saw that when the monoid is represented as a suitably rich collection of functions, then the corresponding $X(M)$ is only linear in very small cases.

Alternatively, in Section \ref{sec:two_other_monoids} we looked at two geometric monoids defined in terms of generators and relations. Having two distinct generators immediately prevents $X(M)$ from being linear, hence $M$ cannot be Ramsey.

\subsection{
	\texorpdfstring{
		Compression of information
	}{
		\ref{subsec:compression}. Compression of information
	}
}
\label{subsec:compression}

It is natural to step back and ask if there is a meta-reason \emph{why} these geometric monoids failed to be Ramsey, but monoids that have been extracted from existing Ramsey theorems have linear $X(M)$.

In a sense, the monoid action on a set $X$ is a compression of the information in an element of $X$; when one acts by elements ``further from the identity'', one loses more information in the process. For example, in the context of Gowers' $\text{FIN}_k$ theorem, the monoid $\{1, T, T^2, \dots, T^k = 0\}$ acts on the set $[k] = \{0, 1, \dots, k\}$ and in turn acts coordinatewise on functions from $\N$ to this set, where $T$ is the usual tetris operation mentioned in Section \ref{subsec:Catalan}. Acting on a function with a ``more complex'' element of this monoid compresses it more, usually cutting off some amount of its support. In a notable application of Gowers' theorem to the Banach space $S_{c_0}$, see \cite{gowers1992lipschitz} or \cite[Theorems 2.22, 2.37]{todorcevic2010introduction}, the space is first discretized to a $\delta$-net of elements of $c_{00}$ that is naturally isomorphic to some $\text{FIN}_k$. After this identification, the tetris operation $T$ is seen to act coordinatewise by scalar multiplication by a number less than 1---a literal compression of basis vectors. 

In a general $X(M)$, the longer the word $w \in M$ is, the smaller $w M$ is. That is, it adds more letters to each word, in turn compressing the set of resulting words. See Figure \ref{fig:complexity}.

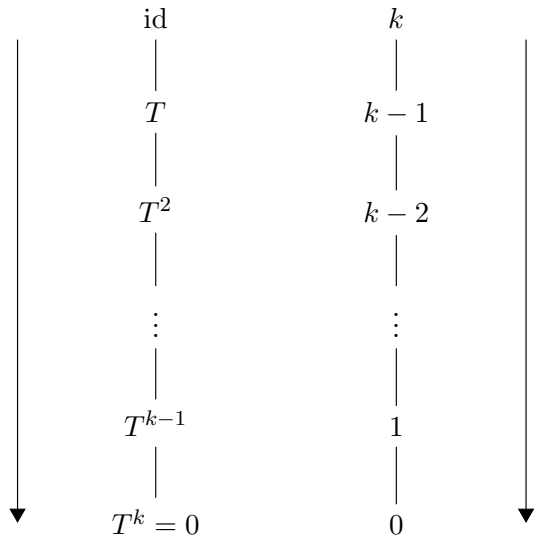
\begin{figure}[h]
\[
\tikzcdset{arrow style=tikz, diagrams={>=triangle 60}}
\begin{tikzcd}
   	{} \arrow[scale=10]{ddddd} & \text{id} \arrow[dash]{d} & & k \arrow[dash]{d} & \, \arrow{ddddd}\\
	{} & T \arrow[dash]{d}         & & k-1 \arrow[dash]{d} & \, \\
	{} & T^2 \arrow[dash]{d}       & & k-2 \arrow[dash]{d} & \, \\
	{} & \vdots \arrow[dash]{d}    & & \vdots \arrow[dash]{d} & \, \\
	{} & T^{k-1} \arrow[dash]{d}   & & 1 \arrow[dash]{d} & \, \\
	{} & T^k = 0   & & 0  & \,  
\end{tikzcd}
\]
\caption{The words of the monoid with $k$ repeated Tetris operations on the left, and the elements $wM$ on the right (suppressing the 
$M$).}
\label{fig:complexity}
\end{figure}

When $X(M)$ is not linear, we might interpret $M$ as containing two or more different measures of compression or complexity. For example, in a monoid defined with two non-commuting generators $a,b$, $a M$ is not comparable to $b M$, so we can think of $a$ and $b$ as being different, incompatible forms of compression. In this setting, acting by each generator causes a loss of different, orthogonal forms of information in $M$.

\bigskip

This soft, heuristic discussion hints at what type of monoids might yield useful applications from Solecki's Ramsey theorem machine; we should look for monoids that measure complexity or compression of information. Monoids like the 0-Hecke monoids and the hyperplane face monoids are more about capturing the geometric notions of reflections. Alternatively, viewing these geometric monoids as a measure of compression in some appropriate way might lead to deeper understandings for reflection groups. 


In order to imitate the above application of Gowers' theorem in the context of monoids, we would need to consider the monoid as a discrete version of a continuous object. In this generalized setting, the monoid action should correspond to a non-commutative scalar multiplication. We provide this ``continuous'' version of the monoid. 

\begin{dfn} Let $M$ be a finite monoid with generators $\{a,b\}$, such that $m = \min\{i > 1 : a^i = a\}$, $n = \min\{i > 1 : b^i = b\}$. Let $C(M)$ be the collection of all finite words of the form:
\[
	(c_1 e_a)(c_2 e_b)(c_3 e_a) \ldots \text{ or } (c_1 e_b)(c_2 e_a)(c_3 e_b) \ldots
\]
where $c_i \in [0,1]$ and
\begin{itemize}
	\item $(c_1 e_x)(c_2 e_x) = (c_1 c_2) e_x$, for $x =a,b$.
	\item $(0 e_x)(e_y) = e_y$, for $x \neq y$.
\end{itemize}

There is an $f : \{m,n\} \rightarrow [0,1]$ such that
\begin{itemize}
	\item $f(m)^{i} \not\approx 0$ for $i < m$, but $f(m)^{m} \approx 0$,
	\item $f(n)^{i} \not\approx 0$ for $i < n$, but $f(n)^{n} \approx 0$.
\end{itemize}
Here ``$\approx 0$'' is meant to be ``approximately 0'', in a suitable context. For example, we can take $f(n) = (1 + \delta_n)^{-i}$, where $0 < \delta_n < 1$ is a number such that $(1 + \delta_n)^{n-1} = \delta$.

Let $M$ act on $C(M)$ by:
\begin{itemize}
	\item $a \cdot (c e_b) = 0 e_b$, and $b \cdot (c e_a) = 0 e_a$,
	\item $a \cdot (c e_a) = f(m) c e_a$, and $b \cdot (c e_b) = f(n) c e_b$.
\end{itemize}
\end{dfn}

In this setting, $C(M)$ is the analogue to $[0,1]$, and (eventually minimal) sequences of elements of $C(M)$ are the analogue to $c_{00}$. Notably, when $M$ is the monoid of the $k$ Tetris operations on $[k]$, then $C(M)$ is precisely the $\delta$-net that appears in the application of Gowers theorem, see \cite[p.37]{todorcevic2010introduction}.

Conversely, if a problem we are considering naturally involves a sequence of relatively simple monoid actions on sets, Solecki's machine may yield worthwhile results, even if the monoid in question is not Ramsey.

The result that so many of these monoids are not Ramsey is still interesting in its own right, as it implies the existence of some interesting colourings.



\bibliographystyle{plain}
\bibliography{references}

\end{document}